\documentclass{article}

\usepackage{amssymb,amsmath,epsfig,fancyhdr,enumerate, caption,anysize, float, graphics,amsfonts,latexsym,amsthm,color}
\usepackage[utf8]{inputenc}
\usepackage{amsmath}

\newtheorem{theorem}{Theorem}[section]
\newtheorem{corollary}[theorem]{Corollary}
\newtheorem{lemma}[theorem]{Lemma}
\newtheorem{proposition}[theorem]{Proposition}

\newcommand{\indicator}[1]{\mathbb{I}_{#1}}

\newcommand\blfootnote[1]{%
  \begingroup
  \renewcommand\thefootnote{}\footnote{#1}%
  \addtocounter{footnote}{-1}%
  \endgroup
}

\title
{On the  number of crossings in a random labelled tree with vertices in convex position}

\author{Octavio Arizmendi, Pilar Cano and Clemens Huemer}

\begin{document}
\maketitle
\begin{abstract}
We prove that the number of crossings in a random labelled tree with vertices in convex position is asymptotically Gaussian  with mean $ n^2/6$ and variance $ n^3/45$. A similar result is proved for points in general position under mild constraints.
\end{abstract}
\blfootnote{\begin{minipage}[l]{0.3\textwidth} \includegraphics[trim=10cm 6cm 10cm 5cm,clip,scale=0.15]{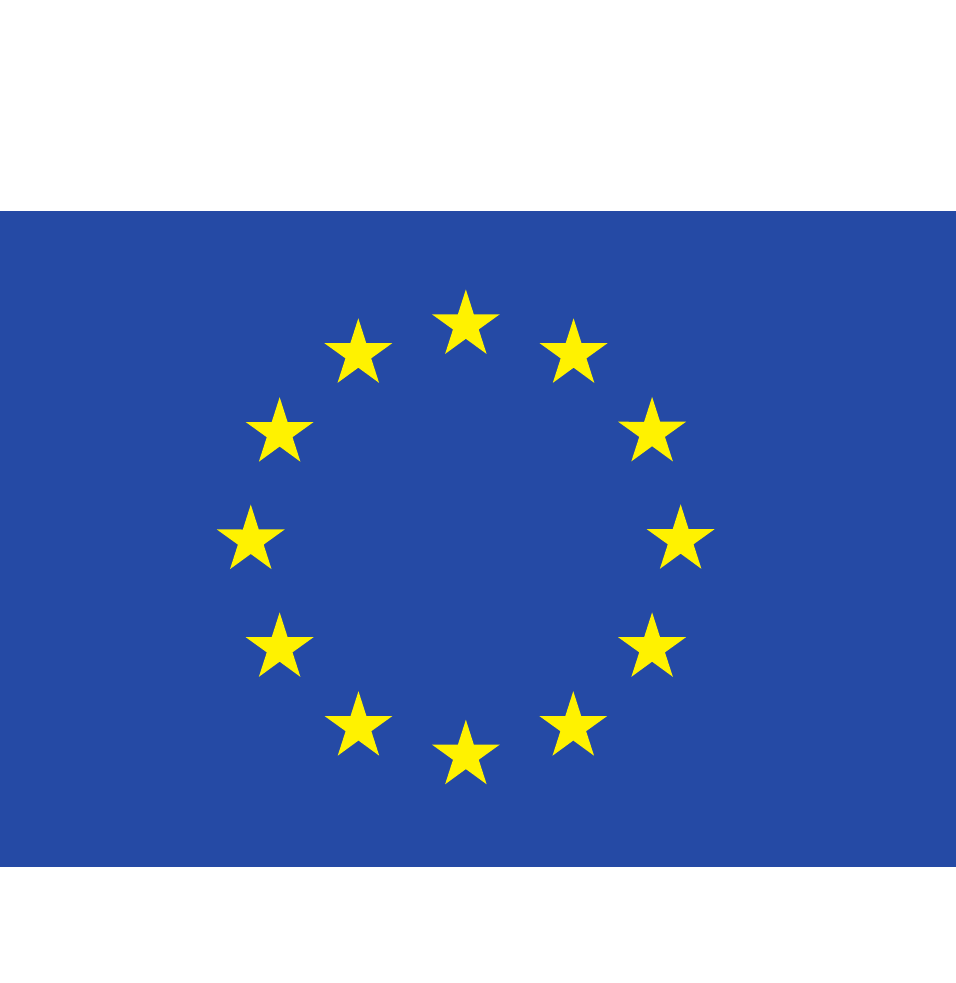} \end{minipage}
 \hspace{-2cm} \begin{minipage}[l][1cm]{0.82\textwidth}
 	  This project has received funding from the European Union's Horizon 2020 research and \\ innovation programme under the Marie Sk\l{}odowska-Curie grant agreement No 734922.
		 	\end{minipage}
			C.H. was supported by projects MINECO MTM2015-63791-R and Gen.Cat. DGR 2017SGR1336.}

\section{Introduction}\label{intro}

Asymptotic poperties of random graphs have been broadly studied since the beginning of the theory. The celebrated papers of Erd\H{o}s and R\'enyi \cite{ER1,ER2,ER3} give the connectivity properties of the now called Erd\H{o}s-R\'enyi model $G(n,M)$ and the Erd\H{o}s-R\'enyi-Gilbert model \cite{Gilbert}, $G(n,p)$. Nowadays, many properties are known for these and other important models, such as the configuration model \cite{Bollobas1980} and the Albert Barab\'asi model \cite{AB}. These include degree distributions \cite{bollobas2001, bollobas1981degree}, number of cycles \cite{wormald1981}, and spectral properties \cite{mckay1981, wigner1958}, among others.  We will not try to give an extensive literature but rather refer to the monographs of Bollob\'as \cite{BollobasBook} and the more recent one by Janson, Luczak and Rucinski \cite{SLR}.

  This paper is concerned with geometric properties of random graphs. Namely, we are interested in  the number of crossings of a rectilinear drawing of a random graph.  A \emph{drawing} of a graph $G=(V,E)$ is a set of points in the plane representing its vertices $V$ and for each pair of vertices $a$ and $b$, a simple continuous arc, denoted $(a,b)$, represents the edge in $E$ connecting the corresponding pair of points. We say that the drawing is \emph{rectilinear} if the edges are represented by the straight segment joining $a$ with $b$.  If a pair of edges in a (rectilinear) drawing of $G$ intersect in an interior point, then the intersection point is defined as \emph{a crossing point} or a (resp. \emph{rectilinear}) \emph{crossing} of such drawing. There are not as many known results of geometric properties such as crossings in random graphs as combinatorial ones. However,  already in the mid 60's, Moon  \cite{moon1965} proved the asymptotic normality of the number of crossings of a complete graph embedded randomly in a sphere. 
  
    Another interest in understanding the crossings for random graphs comes when studying the \emph{crossing number}.  The crossing number of a graph $G$ is the minimum number of edge crossings among all the possible drawings of $G$. Computing the crossing number of a graph is known to be difficult, even for complete graphs. In fact, it has been proven by Garey and Johnson~\cite{garey} that knowing whether the crossing number of a graph $G$ is at least a constant $k$ is \textbf{NP}-complete.  An approach for solving the crossing number type problems, starting from the well known Crossing Lemma~\cite{ajtai1982crossing}, is computing the expected value of crossings in random graphs \cite{mohar2011expected}.   Interesting results regarding the crossing number can be found for instance in the papers of Pach and T{\'o}th \cite{pach1997graphs}, Spencer~\cite{Spencer}, and Spencer and T{\'o}th \cite{ST}.

In this paper we consider the number of crossings of a uniform random labelled tree with vertices in convex position. Our approach is combinatorial by using the method of moments (See Lemma \ref{MomentMethod}). Our main technical tool is the combinatorial approach to cumulants as we describe in Section 2. 

Our main theorem, Theorem \ref{T1} below, is very similar in flavor to the result given by Flajolet and Noy~\cite{FN}, where they consider a random perfect matching on $2n$ points in convex position. They show that the number of rectilinear crossings of this random matching follows asymptotically a Gaussian distribution. However their method of proof is quite different since they use an analytical approach based on an ad-hoc integral representation for a $q$-series. We must mention that they hint of some combinatorial aspects for the first moments,  but do not continue this direction for all moments, because in the words of Flajolet and Noy, ``the combinatorics for higher moments soon become intractable".  In the case of Theorem \ref{T1}, keeping track of the moments directly seems also to be unfeasible. However,  in this case, it turns out that analyzing the asymptotic behavior of cumulants is possible.

The main theorem of this paper is the following.

\begin{theorem} \label{T1}
Let, for each $n$, denote by $S_n$ a set of $n$ points in convex position in the plane. Then as $n$ tends to infinity, the number of (rectilinear) edge crossings of a random tree drawn at random on $S_n$, $X_n$, approaches a normal distribution with mean $\approx n^2/6$ and variance $\approx n^3/45$. 
In other words $$\frac{X_n-\frac{n^2}{6}}{\sqrt{n^3/45}}\to N(0,1)$$
in distribution.\end{theorem}

 Moreover,  the method of proof used in Theorem \ref{T1} can be easily generalized to points which may not be in convex position, provided some constraints on the asymptotic behavior of the \emph{rectilinear crossing number} $\overline{{cr}}$ of the set of points is satisfied (see Section 4, for definition of $\overline{{cr}}$). 
 
\begin{theorem}\label{T2}
Let $\{S_n\}_{n=1}^{\infty}$ be a sequence of point sets in general position in the plane, with $|S_n|=n$, such that  $\lim_{n \rightarrow \infty} \frac{\overline{{cr}}(S_n)}{{{n} \choose {4}}}$ exists. Then, as $n$ tends to infinity, the number of edge crossings of a random spanning tree drawn at random on $S_n$, approaches a normal distribution with mean $\frac{4\overline{{cr}}(S_n)}{n^2}$ and variance $O(n^3)$. 
\end{theorem}

We expect that the analogs of Theorem \ref{T1} and Theorem \ref{T2} also hold more generally. However, it is not straightforward to modify our methods  to other families of graphs.
\section{Preliminaries on Moments and Cumulants}

We first explain the necessary background on moments and cumulants. 

\subsection{Moments}

A family of random variables $\{X_i\}_{i=1}^n$ in a probability space $(\Omega,\mathcal{F},\mathbb{P})$ is said to have finite moments, if $\mathbb{E}[|X_i^{k}|]<\infty$, $k\in\mathbb{N}$.
  
A sequence of random variables $\{Y_n\}_{n>0}$ with finite moments is said to converge $\emph{in moments}$ to $Y$, as $n\to\infty$, if
$$\mathbb{E}[Y_n^k]\to\mathbb{E}[Y^k] , \text{for all }k\geq0.$$

A random variable $X$ is said to be determined by moments if $E[X^n]=E[Y^n]$ for all $n$ implies that $X$ and $Y$ have the same distribution. 

In this paper we will consider the convergence to a standard normal (or gaussian) random variable $Z$. Since the moments of $Z$ are given by  $E[Z^{2n+1}]=0$ and $E(Z^{2n})=(1)(3)\cdots (2n-1)$, then $Z$ is determined by moments, which may be verified by Carleman's criterion.

The importance of determination by moments is the following well known lemma, known as the \emph{Method of moments}.
\begin{lemma}[Method of moments] \label{MomentMethod}
Let X be a random variable which is determined by moments and let $\{X_n\}$ be sequence of random variables with finite moments. If $X_n\to X$ in moments,  then $X_n\to X$ in distribution.
\end{lemma}

     The important notion of independence can be characterized with the use of \emph{joint} moments.  For a family of random variables $\{X_i\}_{i=1}^n$  with finite moments,  the joint moments are the quantities 
$$\mathbb{E}[X_{i_1}X_{i_2} \cdots X_{i_n}]$$ 
which by H\"older's inequality are finite too.

   The random variables $\{X_i\}_i^n$ are independent  if, for any $n_1,n_2,\cdots,n_k \in \mathbb{Z}^+$, the following factorization of moments holds, 
$$\mathbb{E}[X_1^{n_1}X_2^{n_2}  \cdots X_n^{n_k}]=\mathbb{E}[X_1^{n_1}]\mathbb{E}[X_2^{n_2}]\cdots\mathbb{E}[  X_n^{n_k}].$$

\subsection{Cumulants}
 
While the moments are very useful, in this paper we will rather use a variant  of them which, known as cumulants, behaves better when considering sums of independent random variables as will be the case for us.  For a combinatorial approach to cumulants, see  \cite{rota2000combinatorics} and references therein.

In order to do this we need to use partitions $P(n)$.  We call $\pi =\{V_{1},...,V_{r}\}$ a \textbf{partition }of the set $[n]:=\{1, 2,\dots, n\}$
if $V_{i}$ $(1\leq i\leq r)$ are pairwise disjoint, non-void
subsets of $[n]$, such that $V_{1}\cup V_{2}...\cup V_{r}=\{1, 2,\dots, n\}$. We call $
V_{1},V_{2},\dots,V_{r}$ the \textbf{blocks} of $\pi $. The number of blocks of 
$\pi $ is denoted by $\left\vert \pi \right\vert $.


Joint cumulants of $n$ random variables $Z_1,\ldots, Z_n$ denoted by $\{C_k(Z_{i_1},\ldots,Z_{i_k})\}^\infty_{k=1},$  for $\{i_1,...\i_n\}\in[n]^k$, are defined implicitly by the moment-cumulants formula
$$ \mathbb{E}[Z_{i_1}\cdots Z_{i_n}]=\sum_{\pi\in P(n)} C_\pi (Z_1,Z_2,\dots,Z_n),$$
which is, by M\"obius Inversion, equivalent to
$$C_k\left(Z_{i_1},Z_{i_2}, \ldots, Z_{i_k}\right) = \sum_{\pi}(|\pi|-1)!(-1)^{|\pi|-1}\prod_{B \in \pi} \mathbb{E}\left(\prod_{i_j \in B} Z_{i_j}  \right).$$
We denote by $C_k(Z)$ the univariate $k$-th cumulant of a random variable $Z$ given by 
\begin{equation}\label{eq:prop3a}
  C_k(Z):=C_k(Z,Z,\ldots,Z).
\end{equation}
For our purposes the precise formula for cumulants will not be needed but only the following properties of cumulants.
\begin{itemize}
\item (Invariance under shifts). For $k\geq 2$ and any constant $c$, 
\begin{equation}\label{eq:prop1}
  C_k(Z+c)=C_k(Z).
\end{equation}
\item (Homogeneity). For $k\geq 1$ and any constant $\lambda$, 
\begin{equation}\label{eq:prop2}
  C_k(\lambda Z)=\lambda^k C_k(Z)
\end{equation}

\item(Multilinearity). For all $k\geq1$
\begin{equation}\label{eq:prop3}
 C_k(Z_1,\ldots,X_i+Y_i,\ldots Z_n)= C_k(Z_1,\ldots,X_i,\ldots,Z_n)+C_k(Z_1,\ldots, Y_i,\ldots,Z_n)
\end{equation}

\item (Vanishing of mixed cumulants) If for some $1 \leq i,j \leq n$, $Z_i$ and $Z_j$ are \text{independent} random variables then
\begin{equation}\label{eq:prop4}
  C_k(Z_1,Z_2,\ldots,Z_n)=0. 
\end{equation}
\end{itemize}

Finally, the following well-known fact is essential for our proof: A random variable $X$ has normal distribution if and only if $C_k(X)=0$ for all $k \geq 3$.

\section{Asymptotic distribution of number of crossings in a random tree for point sets in convex position}\label{sec:cumulants}

In this section we give the proof of the main theorem. To do this we first recall some properties of the number of trees containing a fixed forest as a subgraph.

\subsection{Probability of containing a fixed forest}\label{sec:count}
Let $a_1, a_2, \ldots, a_p$ be fixed disjoint subtrees on a point set $S$ of size $n$.  Let $E$ be the set of edges defined by the forest $a_1, a_2, \ldots, a_p$.  Let $v_i$ be the number of vertices in sub-tree $a_i$ for all $i \in \{1, \ldots p\}$.

The following result is well known (see e.g. \cite{lovasz}). We give a proof for the convenience of the reader.
This proof is based on Pitman's technique \cite{Pitman} which consists of counting in two different ways sequences of trees containing the $a_1, a_2, \ldots, a_p$ subtrees. 
 \begin{proposition} Let $T(S)$ be the number of trees in $S$ containing the subtrees $a_1, a_2, \ldots a_p$. Then \begin{align*}T(S)=n^{n-|E|-2} \prod_{i=1}^p v_i.\end{align*} \end{proposition}
 \begin{proof}
 First, choose one tree $T$ from the $T(S)$ possible trees containing $a_1, \ldots, a_p$ as subtrees. Now, choose one of the $n$ vertices from $S$ as a root, and from each vertex as a root we get $(n-|E|-1)!$ different sequences for adding the remaining edges of $T$ that do not belong to the fixed subtrees. Hence we have that there are  \begin{align}\label{eq:first}n(n-|E|-1)!T(S)\end{align} ways of choosing $T$.

 On the other hand, the sequence can be constructed as follows. Let the subtrees $a_1, a_2, \ldots a_p$ be fixed. Choose a root vertex for each subtree and orient each edge towards the root. Notice that for each subtree $a_i$ there are $v_i$ different orientations in order to get a rooted subtree. Thus, there are $\prod_{i=1}^p v_i$ different combinations of the fixed subtrees. Now, Pitman's algorithm starts with $n-|E|$ rooted subtrees, i.e., if a vertex $v$ from $S$ does not belong to one of the fixed subtrees, then $v$ itself is a rooted subtree with root $v$. In each step add a new oriented edge to the forest until a rooted tree is obtained in the following fashion. Assume there are $k$ trees and $n-k$ edges, then by choosing any of the $n$ vertices, say $v$, add an edge from $v$ to a root of the remaining $k-1$ rooted subtrees. Then, there are $n(k-1)$ ways of adding a new edge at each step. Therefore, the total number of choices for building a rooted subtree is \begin{align}\label{eq:second}\left(\prod_{k=2}^{n-|E|}n(k-1)\right)\left(\prod_{i=1}^pv_i\right)&=n^{n-|E|-1}(n-|E|-1)!\left(\prod_{i=1}^pv_i\right)\end{align}
 Now, using (\ref{eq:first}) and (\ref{eq:second}) it follows that, \begin{align*} T(S) = n^{n-|E|-2}\prod_{i=1}^p v_i, \end{align*} as desired.
 \end{proof}

 Since the number of labelled trees on an $n$ point set is given by $n^{n-2}$, the probability that a random tree contains a certain forest is the following,
  \begin{equation}\label{eq:prob}
 \mathbb{P}(a_1,a_2, \ldots, a_p)=\frac{n^{n-|E|-2}}{n^{n-2}} \prod_{i=1}^p v_i=n^{-|E|} \prod_{i=1}^p v_i
\end{equation}

Next corollary is an immediate consequence of~(\ref{eq:prob}).
 
\begin{corollary}\label{cor:indep}Consider a random tree $T(S)$ on an $n$ point set $S$ and let let $f_1$ and $f_2$ be two fixed forests on points of $S$ that share no vertices.
Let $\indicator{f_1}$ and $\indicator{f_2}$ be the indicator random variables 
of $T(S)$ containing $f_1$, respectively $f_2$.
Then, $\indicator{f_1}$ and $\indicator{f_2}$ are independent.
\end{corollary}


\subsection{Expectation and variance}\label{sec:expect}

Before proving the main theorem, we calculate 
the normalizing constants, then mean and variance.

That is, if we denote by  $X_n$ the random variable that counts the number of crossings in a random tree on $S$, we want to calculate  $\mathbb{E}(X_n)$  and  $Var(X_n)=\mathbb{E}(X_n^2)-\mathbb{E}(X_n)^2$.

For convenience we identify $S$ with the $n$ first positive integers, i.e. $S=[n]:=\{1,2,3,...,n\}$. In this sense, one crossing may be encoded by two edges $(a,b)$, $(c,d)$  with $a<c<b<d$.

With this in mind, let $\mathbb{I}_{ab}$ be the indicator of the event that edge $(a,b)$ appears in a tree. Thus, we write 
\begin{equation}\label{sum of indicators} X_n= \displaystyle\sum_{a<c<b<d} \indicator{ab}\indicator{cd}.\end{equation}

\begin{proposition}\label{prop:expect}The expectation of the number of crossings  in a random tree on $S$ is given by
\begin{align*} \mathbb{E}(X_n) = \frac{4}{n^2}{n \choose 4} = \frac{(n-1)(n-2)(n-3)}{6n}.\end{align*}
\end{proposition}

\begin{proof}
There are ${n \choose 4}$ ways to choose four points $a,b,c,d$ from $S$, and for each such choice there is one product of indicator variables $\indicator{ab} \indicator{cd}$ with $a<c<b<d.$  By~(\ref{eq:prob}), $\mathbb{E}(\indicator{ab}\indicator{cd})=\mathbb{P}(\indicator{ab}\indicator{cd}=1) = \frac{4}{n^2}$.  The result follows by linearity of expectation.
\end{proof}


\begin{proposition}
$$
Var(X_n)=\frac{n^3}{45}-\frac{3n^2}{40}-\frac{17n}{72}+\frac{35}{24}-\frac{1003}{360n}+\frac{157}{60n^2}-\frac{1}{n^3}
$$
\end{proposition}

\begin{proof}

To calculate the variance, we need the second moment, 

$$\mathbb{E}(X_n^2)=\sum_{\substack{a<c<b<d\\ e<g<f<h}} \indicator{ab} \indicator{cd} \indicator{ef} \indicator{gh}.$$
Note that points of $S$ which give rise to the indicator variables $\indicator{ab}$ and $\indicator{cd}$ might also appear in the indicator variables 
$\indicator{ef}$  or $\indicator{gh}$. This repetition of points leads to different cases of possible crossing configurations, 
depicted in Figure~\ref{fig:treecases}. The configuration shown on the top left in the figure corresponds to the case when $a,b,c,d,e,f,g,h$ are all different points. Such a configuration appears $\frac{ {{n}\choose{4}}{{n-4}\choose{4}}  }{2}$ times (choose four points for the first crossing, another four points for the second crossing, and then we counted each pair of crossing edges twice). By~(\ref{eq:prob}), the probability of this crossing configuration is $n^{-4}2^4$. A similar but a bit tedious argumentation can be done for the other cases.   
Instead of going through all the possible cases, we observe that $\mathbb{E}(X_n^2)$ has the form 
$\mathbb{E}(X_n^2)=\sum_{i=-4}^{4} a_i n^i$, for some values $a_i$ which we have to determine. 
Indeed, for the different crossing configurations, between five and eight points from $S$ are chosen, for each such choice there is some constant number of products of indicators $\indicator{ab} \indicator{cd} \indicator{ef} \indicator{gh}$ with $a<c<b<d$ and  $e<g<f<h$; further, by~(\ref{eq:prob}), $\mathbb{P}(\indicator{ab} \indicator{cd} \indicator{ef} \indicator{gh}=1)$ is $n^{-4}$ or $n^{-3}$, multiplied with some constant.  
With the aid of a computer we calculated the second moment of 
$X_n$ for the first ten values of $n$, see Tables~\ref{tab:crossings} and~\ref{tab:moment}.  
Using these first nine values of $\mathbb{E}(X_n^2)$ we get a linear system with nine variables $a_i$. Its solution gives the following formula for $\mathbb{E}(X_n^2)$.
$$\mathbb{E}(X_n^2)=\frac{n^4}{36}-\frac{14n^3}{45}+\frac{553n^2}{360}-\frac{305n}{72}+\frac{491}{72}-\frac{2323}{360n}+\frac{217}{60n^2}-\frac{1}{n^3}.$$
Then, the result follows from $Var(X_n)= \mathbb{E}(X_n^2)-\mathbb{E}(X_n)^2$.\end{proof}

\begin{figure}[h]
	\centering
		\includegraphics[scale=0.5]{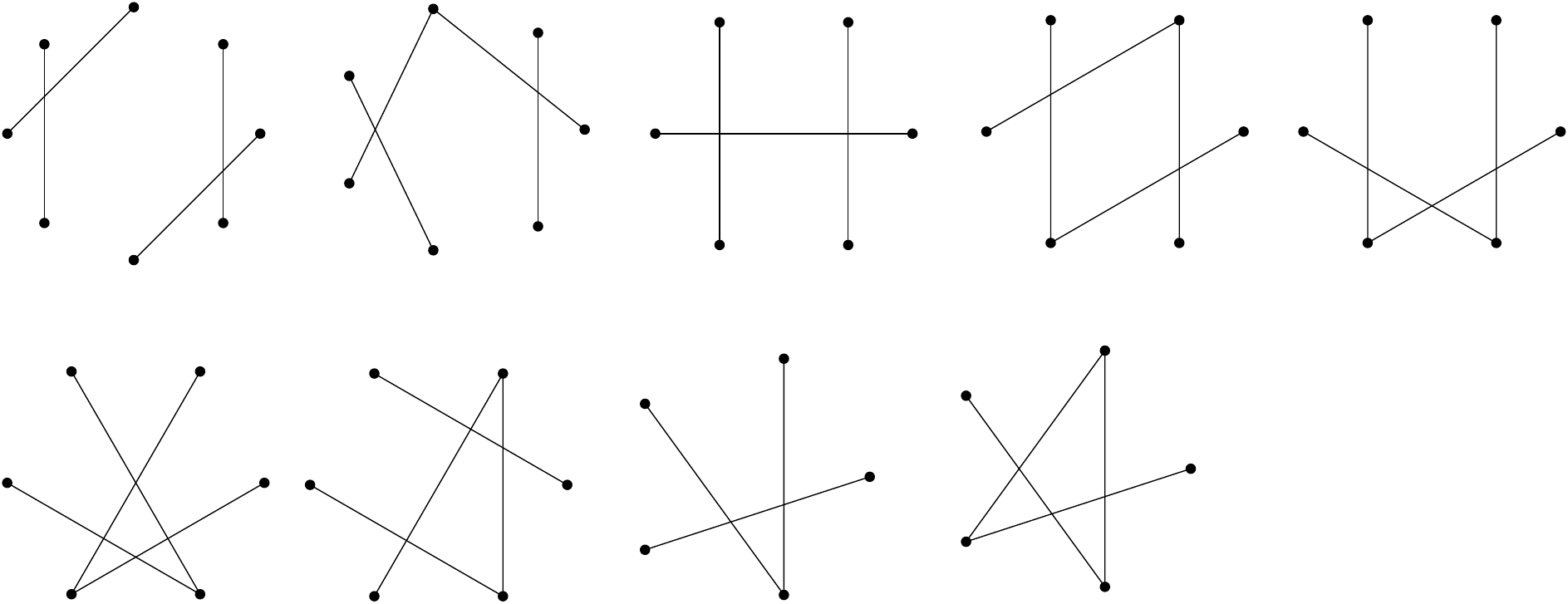}
	\caption{The possible crossing configurations for the case analysis to calculate $\mathbb{E}(X_n^2)$.}
	\label{fig:treecases}
\end{figure}


Notice that, when $n\to\infty$, $E(X_n)\sim n^2/6$ while $Var(X_n)\sim n^3/45.$

\subsection{Proof of Theorem \ref{T1}}

Now we are able to prove the main theorem of the paper. That is, we will show that, as $n \rightarrow \infty,$ 
$$\frac{X_n-\mu_n}{\sigma_n} \rightarrow N(0,1),$$ in distribution, where $\mu_n=E(X_n)$ and $\sigma_n = \sqrt{Var{(X_n)}}.$ To do this, it is sufficient to show that for $k \geq 3,$ $C_k\left(\frac{X_n-\mu}{\sigma}\right) \rightarrow 0$ as $n \rightarrow \infty.$

From Properties~(\ref{eq:prop1}) and~(\ref{eq:prop2}), it is sufficient to show that $\frac{C_k(X_n)}{\sigma^k}  \rightarrow 0$ as $n \rightarrow \infty.$
Since $\sigma = n^{3/2}+o(n^{3/2})$, it is sufficient to show that $C_k(X_n) \in o(n^{3k/2}),$  for $k \geq 3.$ This is our aim. \\

We have from \eqref{sum of indicators}, that $X_n= \sum_{a<c<b<d} \indicator{ab}\indicator{cd}$. To simplify notation, by relabelling, we denote the products of random variables $\indicator{ab}\indicator{cd}$ as $Y_i$, for $i=1,2,\ldots,m={{n}\choose{4}}$.  In this way,
$$C_k(X_n)=C_k\left(\sum_{a<c<b<d} \indicator{ab}\indicator{cd}\right) = C_k\left(\sum_{i=1}^{m} Y_i\right).$$

 Using  multilinearity of cumulants, i.e. property (\ref{eq:prop3}),
\begin{equation}\label{prop:joint}
C_k(X_n)=C_k\left(\sum_{i=1}^{m} Y_i, \sum_{i=1}^{m} Y_i, \ldots,\sum_{i=1}^{m} Y_i \right)= \sum_{\substack{i_j \in \{1,\ldots,m\}\\ \mbox{for} \ j \in \{1,\ldots,k\}}} C_k\left(Y_{i_1},Y_{i_2}, \ldots, Y_{i_k}\right).
\end{equation}

We will analyze which summands are not equal to $0$ in the last formula. For this we use the following notation:

\begin{enumerate}
\item For each $i\in\{1,\dots,m\}$ we denote by $ab(i)$ and $cd(i)$, the edges that correspond to the indicator $Y_i$, (i.e. $Y_i=\indicator{ab(i)} \indicator{cd(i)}$) and define $W_i:=ab(i)\cup cd(i)$.

\item For a set of indices $\mathbf{i}:=(i_1,...,i_k)$ we consider the graph $G(\mathbf{i})=G(i_1,...,i_k)$, which is defined by the union of all the edges $ab(i_r)$ and $cd(i_r)$, for $r=1,...,k$.  That is, $G(\mathbf{i})=\cup_{\ell=1}^{k} W_{i_\ell}.$
\end{enumerate}

 Notice that from Property~(\ref{eq:prop4}) we have that $C_k\left(Y_{i_1},Y_{i_2}, \ldots, Y_{i_k}\right)=0$ if two of the random variables $Y_{i_1},Y_{i_2}, \ldots, Y_{i_k}$ are independent.  From Corollary~\ref{cor:indep} we know that two such random variables $Y_{i_j}=\indicator{ab}\indicator{cd}$ and $Y_{i_\ell}=\indicator{ef}\indicator{gh}$ are independent if $a,b,c,d,e,f,g,h$ are eight different points of $S$. Therefore, $C_k\left(Y_{i_1},Y_{i_2}, \ldots, Y_{i_k}\right) \neq 0$ only if \textbf{any} two of $W_{i_1},W_{i_2}, \ldots, W_{i_k}$ share at least one point of $S$.

~
 \\
We divide the rest of the proof in 3 steps.\\
~
\\
\textbf{Step 1.} Bounds on the number of edges $|E(G(\mathbf{i}))|$ and vertices $|V(G(\mathbf{i}))|$ in the graph $G(\mathbf{i})$.

Consider a set of indices $\mathbf{i}:=(i_1,...,i_k)$ for which  $C_k\left(Y_{i_1},Y_{i_2}, \ldots, Y_{i_k}\right) \neq 0$. $|E(G(\mathbf{i}))|$ is at most $2k$ because each $W_{i_j}$ contributes with at most two edges to $G(\mathbf{i})$.

$|V(G(\mathbf{i}))|$ is at most $3k+1$, because we can draw $G(\mathbf{i})$ by first drawing the four vertices and two edges, $a(i_1)$ and $b(i_1)$ and  then successively adding the vertices and edges of the other $W_j$, for $j=2,\ldots, k$ to the drawing; since each $W_{i_j}$ share at least one vertex with the subgraph defined by $\cup_{\ell=1}^{j-1} W_{i_\ell}$, at most three new vertices are added to the drawing for each $j=2,\ldots, k$.

We will also need to bound $|V(G(\mathbf{i}))|-|E(G(\mathbf{i}))|$. Using the same argument, it is easy to see that $|V(G(\mathbf{i}))|-|E(G(\mathbf{i}))| \leq k+1$; the contribution of the four vertices and two edges of $W_{i_1}$ to $|V(G(\mathbf{i}))|-|E(G(\mathbf{i}))|$ is $2$, and the contribution of each further $W_{i_j}$, for $j=2,\ldots,k$, is at most $1$.\\
~
\\
\textbf{Step 2.} Bound for the joint cumulants of $Y_i$ for a graph $G(\mathbf{i})$.

Recall that joint cumulants can be expressed in terms of joint moments in the following form:
$$C_k\left(Y_{i_1},Y_{i_2}, \ldots, Y_{i_k}\right) = \sum_{\pi}(|\pi|-1)!(-1)^{|\pi|-1}\prod_{B \in \pi} \mathbb{E}\left(\prod_{i_j \in B} Y_{i_j}  \right),$$
where $\pi$ runs through the list of all partitions of $\{1,2,\ldots,k\}$, $B$ runs through the list of all blocks of the partition $\pi$, and $|\pi|$ is the number of parts in the partition.\\

Since $\prod_{i_j \in B} Y_{i_j}$ is a product of indicator variables, using~(\ref{eq:prob}), we have
$$\mathbb{E}\left(\prod_{i_j \in B} Y_{i_j}  \right) =\mathbb{P}\left(\prod_{i_j \in B} Y_{i_j}  \right)=n^{-|E(G(\mathbf{i}))|}|V(G(\mathbf{i}))|,$$
Note that $|V(G(\mathbf{i}))|$ only depends on $k$, and also the number of partitions $\pi$ only depends on $k$. Therefore, there exists a function $f(k)$ such that
$$C_k\left(Y_{i_1},Y_{i_2}, \ldots, Y_{i_k}\right) \leq f(k)n^{-|E(G(\mathbf{i}))|}.$$
\\
\textbf{Step 3.} Bound for the sum in Equation~(\ref{prop:joint}).

We partition all the cumulants $C_k\left(Y_{i_1},Y_{i_2}, \ldots, Y_{i_k}\right)$ into classes according to the number $|V(G(\mathbf{i}))|$ of points of $S$ which appear in $\cup_{j=1}^{k}W_{i_j}$. 

There are at most $g(k)n^{ |V(G(\mathbf{i}))|}$ tuples $\left(i_1,i_2, \ldots, i_k\right)$ with $|V(G(\mathbf{i}))|$ points, for some function $g(k)$. 
To see this, notice that a tuple is determined by the set of vertices $V(G(\mathbf{i}))$ and a collection of subgraphs of size $4$, $W_1,...., W_k$ such that $\cup_j W_{j}=G(\mathbf{i}).$
 For the vertex set $V(G(\mathbf{i}))$ there are  ${{n}\choose{ |V(G(\mathbf{i}))|} } \leq n^{ |V(G(\mathbf{i}))|}$ 
 possibilities.  Since $|V(G(\mathbf{i}))| \leq 3k+1$, the number of choices for  $W_1,...., W_k$ is at most $g(k)={{3k+1}\choose{4}}^{ k}$.

Finally,
$$ \sum_{\substack{i_j \in \{1,\ldots,m\}\\ \mbox{for} \ j \in \{1,\ldots,k\}}} C_k\left(Y_{i_1},Y_{i_2}, \ldots, Y_{i_k}\right)
\leq \sum_{|V(G(\mathbf{i}))|=1}^{3k+1}\sum_{\substack{\mathbf{i}=(i_1,...,i_k) \in [m]^k}} C_k\left(Y_{i_1},Y_{i_2}, \ldots, Y_{i_k}\right)$$
$$\leq (3k+1)g(k)n^{|V(G(\mathbf{i}))|}f(k)n^{-|E(G(\mathbf{i}))|} \leq (3k+1)f(k)g(k)n^{k+1}.$$
Thus, $C_k(X_n) \in O(n^{k+1})$ and  $\lim_{n\rightarrow \infty} \frac{C_k(X_n)}{\sigma^k} =0$, as we wanted to show.

\section{Point sets in general position}

A set $S$ of points in the plane is in general position if no three points of $S$ lie on a common line. 
We show here that our result on the number of crossings in random trees drawn on point sets in convex position extends to random trees drawn on point sets in general position. As we consider a limiting process when the number $n$ of points tends towards infinity, we need to specify a sequence of point sets $\{S_n\}_{n=1}^{\infty}$, with $|S_n|=n$, that satisfies a certain structure. The structure of a point set is often encoded  by the 
{\it{order type}}. 

Recall that the \emph{rectilinear crossing number} of a graph $G$, first introduced  by Haray and Hill~\cite{harary1963number} and  denoted $\overline{{cr}}(G)$, is the minimum number of crossings in any drawing of  $G$ such that its edges are represented by straight line segments. For a given set $S_n$ of $n$ points, we say that the rectilinear crossing number of 
$S_n$, denoted by $\overline{{cr}}(S_n)$, is the number of edge crossings of the complete graph $K_n$, when drawn with vertex set the point set $S_n$ and edges drawn as straight segments. Equivalently, the rectilinear crossing number of $S_n$ is the number of convex quadrilaterals with vertices in $S_n$. It is known that
 $0.37997{{n} \choose {4}}+ {O}(n^3) \leq \overline{{cr}}(S_n) \leq {{n} \choose {4}}$ for any set $S_n$ of $n$ points, we refer to~\cite{AFS}. We consider 
$\lim_{n \rightarrow \infty} \frac{\overline{{cr}}(S_n)}{{{n} \choose {4}}}$, which might not exist for every sequence $\{S_n\}$.  Note that if $\{S_n\}$ describes a sequence of point sets in convex position, then this limit is equal to $1$. It is also known that this limit exists for the sequence $\{S_n\}$ in which  $S_n$ minimizes $\overline{{cr}}(S)$ among all sets $S$ of $n$ points, commonly denoted as $\overline{{cr}}(K_n)$, and this limit is closely related to Sylvester's four point problem~\cite{SW94}.

We show first that the expected number of crossings of a tree drawn at random on a set $S_n$ of $n$ points only depends on $\overline{{cr}}(S_n)$.
Recall that $X_n$ is the random variable that counts the number of crossings in a random spanning tree of $S_n$. 

\begin{proposition}\label{prop:expectgeneral}
\begin{align*} \mathbb{E}(X_n) = \frac{4 \overline{{cr}}(S_n)}{n^2}\end{align*}
\end{proposition}
\begin{proof}
The proof is almost identical to the one of Proposition~\ref{prop:expect}; the difference is that instead of all possible ${n \choose 4}$ edge crossings, we now have $\overline{{cr}}(S_n)$ many. Note that the probability of a given edge crossing to appear in a random tree is $\frac{4}{n^2}$, because all the statements of Section~\ref{sec:count} are invariant of the precise position of the points and of edge crossings, it only matters which points are connected. Hence, formula~(\ref{eq:prob}) applies.
\end{proof}

Let us remark that the variance $Var(X_n)$ does not only depend on $\overline{{cr}}(S_n)$. However, we show in the following that $Var(X_n)$ is of order $O(n^3)$, instead of the theoretically possible $O(n^4)$, which allows us to follow the argumentation of Section~\ref{sec:cumulants}. 

\begin{proposition}
$Var(X_n)$ is in $O(n^3)$. 
\end{proposition}
\begin{proof}
It is sufficient to show that $\mathbb{E}(X_n^2)-\frac{16(\overline{{cr}}(S))^2}{n^4}$ is in $O(n^3)$, using $Var(X_n)=\mathbb{E}(X_n^2)-\mathbb{E}(X_n)^2.$
Define indicator random variables as in Section~\ref{sec:expect}. Then,
$$\mathbb{E}(X_n^2)=\sum_{\substack{a<c<b<d \\ e<g<f<h}} \indicator{ab} \indicator{cd} \indicator{ef} \indicator{gh}.$$
Again, points of $S$ which give rise to the indicator variables $\indicator{ab}$ and $\indicator{cd}$ might also appear in the indicator variables 
$\indicator{ef}$  or $\indicator{gh}$. This repetition of points leads to different cases of possible crossing configurations, 
which depends on the numbers of different sub-order types of $5$, $6$, $7$, and $8$ points. While it seems infeasible to come up with a precise case analysis, we only need to consider products $\indicator{ab} \indicator{cd} \indicator{ef} \indicator{gh}$ with  $a<c<b<d$ and $e<g<f<h$ where 
 $a,b,c,d,e,f,g,h$ are eight different points. Indeed, the number of crossing configurations which involve at most seven points can be upper bounded by ${{n}\choose{7}}$ multiplied with some constant. Thus, 
$$\sum_{\substack{a<c<b<d \\ e<g<f<h}} \indicator{ab} \indicator{cd} \indicator{ef} \indicator{gh}=(cr(S_n))^2-\Theta(n^7).$$
Also, when there are eight different points, by equation~\ref{eq:prob} we have that $\mathbb{P}(\indicator{ab} \indicator{cd} \indicator{ef} \indicator{gh}=1) = 16n^{-4}$. 
Therefore,  $\mathbb{E}(X_n^2)=\frac{16(\overline{{cr}}(S))^2}{n^4}+O(n^3)$, as claimed.\\
\end{proof}

Finally, note that as in Section~\ref{sec:count},  the precise positions of the points of $S$ or edge crossings are not relevant for the arguments of Section~\ref{sec:cumulants}, only the connected components matter. Hence, the results of Section~\ref{sec:cumulants} also apply in this setting.

\subparagraph*{Acknowledgments.} We thank  David Flores and Vincent Pilaud for useful discussions.


\bibliographystyle{plain}
\bibliography{biblio}

\section*{Appendix}

\begin{table}[h]
\begin{tabular}{l|l}
n & \\
\hline
1 & 1\\
2 & 1\\
3 & 3\\
4 & 12, 4\\
5 & 55, 45, 20, 5\\
6 & 273, 378, 321, 204, 78, 36, 6\\
7 & 1428, 2856, 3535, 3430, 2415, 1659, 847, 385, 203, 42, 7\\
8 & 7752, 20520, 33216, 42408, 41936, 38192, 29048, 20280, 13696, 7752, \
4048, 2016, 960, 248, 64, 8\\
9 & 43263, 143451, 286308, 448371, 560124, 629019, 613413, 549162, 462285, 356193, 257121, 176040, 115740,\\
 & 67563, 38538, 19863, 10323, 4275, 1386, 450, 72, 9\\
10 & 246675, 986700, 2339450, 4314890, 6440875, 8531520, 9974515, 10686500, 10686395, 9966550, 8771495, \\
& 7339860, 5890895, 4463120, 3265750, 2269070, 1534005, 982890, 592545, 345720, 190395, 100350, \\
& 49115, 20040, 7480, 2570, 520, 100, 10\\
\end{tabular}
\caption{number of spanning trees on $n \leq 10$ points with  $k$ edge crossings, in increasing order of $k=0,1,2,\ldots$}
\label{tab:crossings}
\end{table}

\begin{table}[h]
\begin{tabular}{l|l l l l l l l l l l}
$n$ & 1& 2 &3&4& 5& 6& 7& 8& 9&10 \\ 
\hline
$\mathbb{E}(X)$& 0 &  0 & 0 & 1/4 &4/5 & 5/3 & 20/7& 35/8 & 56/7& 42/5 \\
$\mathbb{E}(X_n^2)$& 0 &  0 & 0 & 1/4 &34/25 & 977/216 & 3968/343 & 12789/512 & 34916/729 & 42063/500\\
\end{tabular}
\caption{The first two moments of $X_n$ for $n \leq 10$}
\label{tab:moment}
\end{table}
 \textsc{Octavio Arizmendi: Centro de Investigaci\'on en Matem\'aticas, Apdo. Postal 402,
Guanajuato, Gto. 36000, Mexico;  Pilar Cano: School of Computer Science, Carleton University, 1125 Colonel By Dr, Ottawa, ON K1S 5B6, Canada; Pilar Cano and Clemens Huemer: Universitat Polit\'ecnica de Catalunya, Jordi Girona 1-3,
08038 Barcelona, Spain\\}
\emph{E-mail address:} \texttt{octavius@cimat.mx, m.pilar.cano@upc.edu, clemens.huemer@upc.edu}

\end{document}